%% file: Acounting.tex
\documentclass[12pt]{amsart}
\input{style}

\title{Counting Surface Subgroups in Cusped Hyperbolic 3-Manifolds}
\author{Xiaolong Hans Han}
\address{Center for Mathematics and Interdisciplinary Sciences, Fudan University, Shanghai, 200433, China\newline\indent
Shanghai Institute for Mathematics and Interdisciplinary Sciences (SIMIS), Shanghai, 200433, China}
\email{xhh@simis.cn}

\author{Zhenghao Rao}
\address{Department of Mathematics, Rutgers University--New Brunswick, Piscataway, NJ 08854, USA}
\email{zhenghao.rao@rutgers.edu}

\author{Jia Wan}
\address{Department of Mathematics, University of Wisconsin--Madison, Madison, WI 53706, USA}
\email{wan44@math.wisc.edu}

\date{}

\begin{document}

\begin{abstract}
Let $M =\mathbb{H}^3/\Gamma$ be a finite-volume, noncompact hyperbolic $3$-manifold.  We show that the number of quasi-Fuchsian surface subgroups of $\Gamma$ (up to conjugacy and commensurability) of genus at most $g$ is bounded both above and below by functions of the form $(cg)^{2g}$. As a corollary, for all $h\geq 4$, the number of purely pseudo-Anosov closed surface subgroups of genus at most $g$ of the mapping class group $\mathrm{Mod}(S_{h,0})$ is bounded below by $(Cg)^{2g}$ for a universal constant $C$. In contrast, for some $g \geq 2$, we construct infinitely many conjugacy classes of genus-$g$ surface subgroups of $\Gamma$ with accidental parabolics.
\end{abstract}

\maketitle

\input{Introduction}
\input{Upper-new}
\input{Lower}
\input{pA}

\input{Coannular}

\bibliographystyle{alpha}
\bibliography{references} 
\noindent 

\noindent

\end{document}

%% file: style.tex
\usepackage{amsmath, amssymb, amscd, amsthm, amsfonts, mathrsfs, tikz-cd}
\usepackage{graphicx}
\usepackage{indentfirst}
\usepackage{thmtools,mathtools}
\usepackage{thm-restate}
\usepackage[colorlinks, bookmarks=true, backref=page]{hyperref}
\hypersetup{
	colorlinks,
	citecolor=blue,
	linkcolor=blue,
	urlcolor=blue}
\usepackage{cleveref}
\usepackage{cite}
\usepackage{verbatim}
\usepackage{chngcntr}

\usepackage{autonum}

\newtheorem{theorem}{Theorem}[section]
\newtheorem{lemma}[theorem]{Lemma}

\newtheorem{corollary}[theorem]{Corollary}

\newtheorem{lem}[theorem]{Lemma}

\newtheorem{q}[theorem]{Question}

\newtheorem{conj}[theorem]{Conjecture}
\theoremstyle{definition} 
\newtheorem{defi}[theorem]{Definition}
\newtheorem{rem}[theorem]{Remark}

\newtheorem*{remark}{Remark}

\counterwithout{equation}{section}

\newcommand{\al}{\alpha}

\DeclareMathOperator{\foot}{foot}
\DeclareMathOperator{\hl}{\mathbf{hl}}

\newcommand{\poi}{$\pi_1$-injective }
\newcommand{\poin}{$\pi_1$-injective}

\newcommand{\sys}{\tn{sys}}

\newcommand{\psltc}{\tn{PSL}(2,\C)}

\newcommand{\im}{\tn{Im}}

\newcommand{\C}{\mathbb{C}}

\newcommand{\inj}{\textnormal{inj}}

\newcommand{\me}{M_{\epsilon}}

\newcommand{\R}{\mathbb{R}}

\newcommand{\tn}{\textnormal}

\newcommand{\vol}{\textnormal{vol}}
\newcommand{\Z}{\mathbb{Z}}

\DeclareMathOperator{\PSL}{\tn{PSL}}

\newcommand{\gt}{\mathcal{G}_2}

\renewcommand{\H}{\mathbb H}

\newcommand{\from}{\colon}

\usepackage{enumitem}

\setlist{itemsep=3pt plus 1pt minus 1pt}

\makeatletter
\newcommand{\proofpart}[2]{%
  \par
  \addvspace{\medskipamount}%
  \noindent\emph{Part #1: #2}\par\nobreak
  \addvspace{\smallskipamount}%
  \@afterheading
}
\makeatother

\newcommand{\bs}{\backslash}
\newcommand{\TT}{\mathcal T}

%% file: Introduction.tex
\section{Introduction}
\subsection*{Main Results}
A hyperbolic manifold is \emph{cusped} if it is complete, noncompact, and has finite volume. Let $M=\mathbb{H}^3/\Gamma$ be a cusped hyperbolic $3$-manifold and $S_g$ a closed surface of genus $g$ with $g\geq 2$. For a continuous, $\pi_1$-injective map $f\from S_g\to M$, the image $H = f_*(\pi_1(S_g))$ is called a \emph{surface subgroup of} $\Gamma$. A surface subgroup $H$ is \emph{coannular} if $H$ contains an accidental parabolic element, i.e., there exists a non-peripheral $\alpha\in\pi_1(S_g)$ such that $f_*(\alpha)$ is parabolic in $\Gamma$. By the work of Thurston~\cite{thurston1997geometry} and Bonahon~\cite{bonahon1986bouts}, any geometrically finite surface subgroup of $\Gamma$ is either quasi-Fuchsian or coannular. 

For $g\geq2$, let $s(g,M)$ (resp.\ $n(g,M)$) denote the number of conjugacy (resp. commensurability) classes of quasi-Fuchsian surface subgroups in $\Gamma$ of genus at most $g$. Then clearly, $n(g,M)\leq s(g,M)$. Our first result is the following.
\begin{theorem}\label{thm1}
Let $M=\mathbb{H}^3/\Gamma$ be a cusped hyperbolic $3$-manifold. Then there exist positive constants $C_1,C_2$, depending only on $M$, such that for all sufficiently large $g$,
    \begin{equation}\label{main-estimate}
         (C_2g)^{2g}\leq n(g,M) \leq s(g,M)\leq (C_1g)^{2g}.
    \end{equation}
In fact, the constant $C_1$ depends only on the systole and the volume of a compact core of $M$.
\end{theorem}

One application of \Cref{thm1} is to count the number of purely pseudo-Anosov closed surface subgroups in the mapping class group up to commensurability. Taking $M$ to be the figure-eight knot complement, and combining this with the results in \cite{kent2024atoroidal}, we have the following lower bound.

\begin{corollary}\label{cor3}
There exists a positive constant $C$ such that for all $h\geq 4$, the number $p(g)$ of commensurability classes of purely pseudo-Anosov closed surface subgroups of $\mathrm{Mod}(S_{h,0})$ of genus at most $g$ satisfies  
    \begin{equation}
        p(g)\geq (Cg)^{2g}.
    \end{equation}
As a consequence, the number of commensurability classes of closed, atoroidal $S_h$-bundles over surfaces of genus at most $g$ is also bounded below by $(Cg)^{2g}$.
\end{corollary}

The finite number of conjugacy classes of surface subgroups up to genus $g$ in \cite{kahn2012counting} and \Cref{thm1} may be consistent with the finite number of conjugacy classes of elements in $\pi_1(M)$ up to bounded length, but this philosophy fails in general. In a standard Euclidean $3$-torus, we can ``spin" (analogous to Dehn twists) a totally geodesic $2$-torus along an orthogonal direction to obtain infinitely many non-homotopic $2$-tori. Thurston observed this analogy and remarked~[\citenum{thurston1997geometry}, p.~203] that there are likely infinitely many conjugacy classes of coannular surface subgroups of some fixed genus in $\Gamma$ by spinning around boundary tori, which was also discussed in \cite{Baker-Cooper08} for compact surfaces with boundary. We explicitly construct such an infinite collection of coannular closed surfaces, using ideas from Cooper--Long--Reid~\cite{clrEssentialClosedSurfaceCuspedM} and Maskit's Kleinian combination theorem. 

\begin{theorem}\label{thm:coannular}
    Let $M=\mathbb{H}^3/\Gamma$ be a cusped hyperbolic $3$-manifold. 
    There exists $g\geq2$ such that $\pi_1(M)$ contains infinitely many conjugacy classes of coannular surface subgroups of genus $g$.
\end{theorem}
A coannular surface subgroup of $\Gamma$ is not quasi-Fuchsian, since every nontrivial element of a quasi-Fuchsian surface group is represented by a quasi-geodesic and hence cannot be parabolic.

\subsection*{Related Results}
The study of surface subgroups in hyperbolic $3$-manifolds has attracted sustained attention for decades, including but not limited to existence, ubiquity, and counting. In the closed case, the existence and ubiquity are due to Kahn--Markovi\'c~\cite{kahn2012immersing}. On the quantitative side, Masters~\cite{masters2005counting} presented upper and lower bounds of order $g^{cg}$ for the number of conjugacy and commensurability classes of surface subgroups of genus at most $g$ (the lower bound assumes that $M$ contains a self-intersecting totally geodesic surface). Kahn--Markovi\'c~\cite{kahn2012counting} later strengthened this bound to the order $(cg)^{2g}$, also removing the assumption that $M$ contains a totally geodesic surface. Calegari--Marques--Neves~\cite[Theorem 4.2]{CMN} counted conjugacy classes of $(1+\epsilon)$-quasi-Fuchsian surface subgroups and introduced the minimal surface entropy of closed negatively curved $3$-manifolds, giving estimates for this invariant. Kahn--Markovi\'c--Smilga~\cite{kmsGeometricallyRandom} improves the counting for asymptotically geodesic surfaces in a closed hyperbolic $3$-manifold. 

For cusped hyperbolic $3$-manifolds, Cooper--Long--Reid~\cite{clrEssentialClosedSurfaceCuspedM} constructed closed surface subgroups. Masters--Zhang~\cite{masters2008closed} and Baker--Cooper~\cite{baker2015finite} independently proved the existence of quasi-Fuchsian surfaces, while the ubiquity was proved by Cooper--Futer~\cite{cooper2019ubiquitous} and improved to a quantitative version by Kahn--Wright~\cite{kahn2021nearly}, the construction of which will play an important role in our estimate. Recently, Jiang--Vargas Pallete~\cite{jiang2025minimal} generalized the estimate of minimal surface entropy to cusped hyperbolic $3$-manifolds.

\subsection*{Proof Outline}
The proof of the upper bound in \Cref{thm1} follows the outline in \cite{masters2005counting} and \cite{kahn2012counting}, using a special family of triangulations of the topological closed surface $S_g$. Given such a triangulation $\mathcal{T}$, their works show that the homotopy type of a map $f\from S_g\to M$ is determined by the image of vertices of $\mathcal{T}$, if the edges are all mapped to short geodesic segments.
In the cusped case, a quasi-Fuchsian surface group contains no parabolic elements; hence the intersection of the associated immersed surface and the cusps of $M$ is the union of finitely many disks, which does not contribute to the topology of the immersed surface. In addition, since $M$ is aspherical, we only need to focus on the vertices that are mapped into the thick part of $M$, and the estimate follows similarly as in \cite{kahn2012counting}.

For the lower bound in Theorem \ref{thm1}, we will utilize the construction by Kahn--Wright, which provides a \emph{ubiquitous} collection of nearly geodesic immersed surfaces. These surfaces are made out of \emph{good pants} and \emph{good hamster wheels}. We follow the idea of Kahn--Markovi\'c~\cite{kahn2012counting} to assemble two nearly geodesic surfaces along a common non-separating long geodesic, where the bending angle between these two surfaces is close to $\pi/2$. Kahn--Markovi\'c showed that if the surfaces are made out of good pants, then the new surface is still quasi-Fuchsian and also maximal (i.e., not the covering space of another immersed surface) in $M$. We generalize the result to surfaces constructed by Kahn--Wright, and then obtain a large family of maximal quasi-Fuchsian surface subgroups, which yields the lower bound.

Cooper--Long--Reid~\cite{clrEssentialClosedSurfaceCuspedM} first established the existence of closed surface subgroups in a cusped hyperbolic 3-manifold $M$, by maximally compressing \emph{tubed} surfaces constructed from some properly embedded surfaces with boundary in some $\Z$-covers of some finite cover of $M$. Their tubing construction glues two identical copies of properly embedded essential surfaces $F_1, F_2$ with boundary with annuli of the \emph{same} height. In contrast, our approach joins the boundaries of $F_i$ together using annuli whose heights are chosen inductively (and hence may differ from one another). By carefully analyzing the limits sets and applying the Maskit combinations theorems, we demonstrate that the obtained surface is \poin. 
We then spin this coannular surface around the \emph{last} cusp in the inductive construction and obtain infinitely many coannular surfaces with the same genus in a finite cover of $M$. It turns out that they are all incompressible and pairwise non-homotopic. After projecting down to $M$, only finitely many of them can be homotopic, and then Theorem \ref{thm:coannular} follows.

\subsection*{Future Work.} For a closed hyperbolic 3-manifold $M$, Kahn--Markovi\'c~\cite[Conjecture 1.1]{kahn2012counting} conjectured that there exists a constant $C(M)$ such that both $\sqrt[2g]{n(g,M)}/g$ and $\sqrt[2g]{s(g,M)}/g$ converge to $C(M)$ when $g\to\infty$. 

The existence of surface subgroups in closed hyperbolic manifolds of dimensions at least 4 is established by Hamenst\"adt~\cite{Ham15} and Kahn--Rao~\cite{KR25}. Using the technique developed by Kahn--Wright~\cite{kahn2021nearly}, one should be able to show that any hyperbolic $n$-manifold with finite volume has surface subgroups without accidental parabolics. It is then natural to count the number of conjugacy classes and commensurability classes of surface subgroups in both closed and cusped cases. We make the following conjecture.

\begin{conj}
    For a complete hyperbolic $n$-manifold $M=\H^n/\Gamma$ of finite volume, there exists a constant $C(M)$, such that 
    \begin{equation}
        \lim_{g\to\infty}\frac{1}{g}\sqrt[2g]{n(g,M)}
        =\lim_{g\to\infty}\frac{1}{g}\sqrt[2g]{s(g,M)}
        =C(M).
    \end{equation}
\end{conj}

For higher dimensional cusped hyperbolic manifolds, one can also consider surface subgroups with accidental parabolics.

\begin{q}
    Does every cusped hyperbolic manifold of dimension at least $4$ contain infinitely many conjugacy classes of coannular surface subgroups of the same genus?
\end{q}

Although \Cref{cor3} presents a lower bound for the number of commensurability classes of purely pseudo-Anosov closed surface subgroups in mapping class groups, it is unclear to authors if an upper bound of the same order would hold, since the mapping class groups have more complicated structures in many senses than the fundamental group of the figure-eight knot complement.

\subsection*{Organization.} 
In Section \ref{sec:upper}, we prove the upper bound in Theorem \ref{thm1}. In Section \ref{sec:lower}, we recall the construction of quasi-Fuchsian surfaces by Kahn--Wright~\cite{kahn2021nearly} and prove the lower bound in Theorem \ref{thm1}.
In Section \ref{sec:MCG}, we apply Theorem \ref{thm1} to the figure-eight knot complement and provide the lower bound in Corollary \ref{cor3}.
In Section \ref{sec:coannular}, we prove Theorem \ref{thm:coannular}.

\subsection*{Acknowledgments.} 
We would like to thank Nathan M. Dunfield, Autumn Kent, Alan W. Reid, and Hongbin Sun for many fruitful discussions on surface subgroups of cusped hyperbolic 3-manifolds. We are grateful to Hongbin Sun for his careful reading of the paper and many helpful comments. We thank Chris J. Leininger for helpful discussions on atoroidal surface bundles. We also thank Ben Lowe for valuable comments. The third author would like to express her sincere gratitude to her advisor Autumn Kent, who is always patient and shares inspiring ideas in their weekly discussions. This work was partially completed when the first author visited Rutgers University–-New Brunswick and he would like to thank the institute for its hospitality. 
X.H.H. is partially supported by the start-up grants at Shanghai Institute of Mathematics and Interdisciplinary Sciences and NSFC No. 12501084. The third author was partially supported by NSF Grant DMS-2202718.

%% file: Upper-new.tex
\section{The Upper Bound} \label{sec:upper}
In \cite[Section 8.8]{thurston1997geometry}, Thurston proves that there are only finitely many conjugacy classes of non-coannular surface subgroups of bounded genus in a finite-covolume Kleinian group. Our first goal is to extend the upper bound results in \cite{masters2005counting,kahn2012counting} for quasi-Fuchsian surface subgroups to the cusped case. 

\begin{defi}
    Let $S_g$ be a closed topological surface of genus $g$. A connected graph $\tau$ is called a \emph{$(k,g)$-triangulation}, if it satisfies the following:
    \begin{itemize}
        \item $\tau$ can be embedded into $S_g$ such that every component of the set $S_g\bs \tau$ is a triangle;
        \item each vertex of $\tau$ has degree at most $k$;
        \item the graph $\tau$ has at most $kg$ vertices and at most $kg$ edges.
    \end{itemize}
    The set of all $(k,g)$-triangulations is denoted by $\TT(k,g)$.
\end{defi}

We note that any $(k,g)$-triangulation $\tau$ admits a unique embedding into $S_g$ (up to homeomorphism of $S_g$). 

A Riemann surface $S$ is \emph{$s$-thick} if its injectivity radius satisfies $\inj(S)\geq s$ for some $s>0$. 
The \textit{systole} of $M$ is the length of a shortest closed geodesic in $M$, denoted by $\sys(M)$. Similarly, we define $\sys(S)$ for any closed hyperbolic surface $S$.

For a quasi-Fuchsian surface subgroup given by a $\pi_1$-injective immersion $f\colon S_g\rightarrow M$, there is a hyperbolic structure on $S_g$ (in the conformal class of the pullback metrics), denoted by $S$, and a pleated map $g\from S\rightarrow M$ that is homotopic to $f$.

\begin{lem}\label{lem8}
	Let $f\colon S \rightarrow M$ be a pleated quasi-Fuchsian surface. Then 
	$$2\cdot\inj(S)\geq \sys(M).$$ 
\end{lem}
\begin{proof}
A pleated map $f\colon S\rightarrow M$ is distance non-increasing. Since $f$ has no accidental parabolics, the closed geodesics of $S$ are mapped to nontrivial closed curves in $M$ whose lengths are at least $\sys(M)$. Then the inequality holds since $2\cdot \inj(S)=\sys(S)$.
\end{proof}

\begin{rem}
Kahn--Wright~\cite{kahn2021nearly} proves that there exists a sequence of asymptotically geodesic surfaces $S_i$ in $M$. 
	\cite{assal2025asymptotic,hanNearlyGeodesicFilling} prove that the $S_i$ are asymptotically dense in the Grassmann bundle $\gt M$, and hence become arbitrarily deep inside the cusps of $M$. 
\end{rem}

Let $s=\sys(M)/2$ denote half the systole of $M$. Now suppose $S$ is a pleated quasi-Fuchsian surface and hence $s$-thick.
By \cite[Lemma 2.1]{kahn2012counting}, there exists $k=k(s)>0$ such that $S$ has a $(k,g)$-triangulation $\tau$ that is embedded into $S$ such that every edge of $\tau$ is a geodesic arc of length at most $s$.

Let $\epsilon=\min\{s,\epsilon_3\}$, where $\epsilon_3$ is the Margulis constant. Consider the thick-thin decomposition of $M$ as the union of $M_\epsilon$ (the subset of points whose injectivity radius is at least $\epsilon$) and its complement $M_{<\epsilon}$. With this choice of $\epsilon$, $M_{<\epsilon}$ consists of the cusps of $M$, and $M_\epsilon$ is a compact core homotopy equivalent to $M$. 

We count the upper bound on the conjugacy classes of quasi-Fuchsian surface subgroups by considering the image of vertices of $\tau$ under $f$.
Let $\mathcal{C}=\{C_1,\ldots,C_m\}$ be a minimal finite collection of balls of radius ${\epsilon}/{4}$ forming a cover of $M_\epsilon$. Suppose $h_i\colon S_i\rightarrow M$, $i=1,2$, are two pleated maps from two genus $g$ hyperbolic surfaces $S_i$, and $\tau_i$ the $(k,g)$-triangulation of $S_i$ from \cite[Lemma 2.1]{kahn2012counting}.
We say that $h_1$ and $h_2$ are homotopic if there is a map $q \colon S_1 \to S_2$ such that $h_2\circ q$ is homotopic to $h_1$. In the following lemma, we show that the homotopy class of the surface is essentially determined by those vertices sent into the compact core $M_\epsilon$.

\begin{lemma}\label{lemma}
    Suppose that there is a map $q\from S_1\rightarrow S_2$ such that $q(\tau_1) = \tau_2$. If for any $v\in \tau_1$, $h_1(v)$ and $h_2(q(v))$ are either in the same ball of $\mathcal{C}$ or in $M_{<\epsilon}$, then $h_1$ and $h_2$ are homotopic. 
\end{lemma}
\begin{proof}
    If all the vertices of $\tau_1$ are mapped into $\mathcal{C}$, the proof follows from [\citenum{masters2005counting}, Lemma 2.4]. 

    Suppose some vertices of $\tau_1$ are mapped into $M_{<\epsilon}$. Let $\tau'_i$ be the subgraph of $\tau_i$ induced by the vertices whose images lie in some $C_j\in\mathcal{C}$, i.e., containing all edges of $\tau_i$ between such vertices. Since each $h_i(S_i)$ intersects the cusps of $M$ in finitely many disks, each component of $S_i\bs\tau'_i$ is simply connected. So the induced injection $(h_i)_*\from \pi_1(S_i)\to\pi_1(M)$ is determined by the restriction of $h_i$ to $\tau'_i$,and hence determines the homotopy class of $h_i$. By \cite[Lemma 2.4]{masters2005counting}, $h_1|_{\tau'_1}$ is homotopic to $h_2|_{\tau'_2}$. Therefore $h_1$ is homotopic to $h_2$.    
\end{proof}

The next lemma is recalled from \cite{kahn2012counting}, providing an upper bound for the number of $(k,g)$-triangulations.

\begin{lem}[\citenum{kahn2012counting}, Lemma 2.2]\label{lem:triangulation}
    There exists a constant $C>0$ depending only on $k$, such that for $g$ sufficiently large, we have
    \begin{equation}
        |\TT(k,g)|\leq (Cg)^{2g}.
    \end{equation}
\end{lem}

Now we can present an upper bound for $s(g,M)$ when $g$ is sufficiently large.

\begin{theorem}\label{thm8}
Let $M$ be a cusped hyperbolic $3$-manifold. Then there exists a constant $C_1$ depending only on $\sys(M)$ and $\vol(\me)$, such that for $g$ sufficiently large,
\begin{equation}
    s(g,M)\leq (C_1g)^{2g}.
\end{equation}
\end{theorem}
\begin{proof}

Notice that the constant $k$ associated with the triangulation $\tau$ depends only on $s$ and hence only on the systole of $M$.

Suppose $f\from S\rightarrow M$ is a pleated quasi-Fuchsian surface of genus $g$ with a $(k,g)$-triangulation $\tau$ as before.
For each vertex $v\in\tau$, there are $(m+1)$ choices to map $v$ into a ball in $\mathcal{C}$ or into $M_{<\epsilon}$. Hence by Lemma \ref{lemma}, there are at most $(m+1)^{kg}$ ways to map $\tau$ into $M$ up to homotopy.

Let $\tilde{s}(g,M)$ be the number of $\pi_1$-injective immersed quasi-Fuchsian surfaces of genus $g$ up to homotopy. Then by Lemma \ref{lem:triangulation},
\begin{equation}
    \tilde{s}(g,M)\leq (m+1)^{kg}|\TT(k,g)|\leq (m+1)^{kg}(Cg)^{2g}\leq (c_1g)^{2g},
\end{equation}
for some constant $c_1$ and sufficiently large $g$. The constant $c_1$ then only depends on $C$ (and then only on $s$ and $\vol(\me)$). Then we have
\begin{equation}
    s(g,M)=\sum_{r=2}^g \tilde{s}(r,M)\leq \sum_{r=2}^g (c_1r)^{2r}\leq (C_1g)^{2g},
\end{equation}
for some constant $C_1$, which depends only on $\sys(M)$ and $\vol(\me)$.
\end{proof}

%% file: Lower.tex
\section{The Lower Bound}\label{sec:lower}

In this section, we establish the lower bound in \Cref{thm1}. We will assemble nearly geodesic quasi-Fuchsian surfaces constructed by Kahn--Wright to construct maximal surface subgroups, following the idea of Section 4 of \cite{kahn2012counting}. First, we review the construction by Kahn-Wright for closed surfaces.
We then present the gluing pattern to cut and paste two surfaces along a common good geodesic, where the new surfaces are \emph{maximal}. Finally, we estimate a lower bound on the number of surfaces from such construction, establishing \eqref{main-estimate}. The parameters $\epsilon$ and $R$ will appear frequently in this section, where $\epsilon$ is sufficiently small and positive, and $R$ is chosen to be a sufficiently large integer depending on $\epsilon$.  Once $\epsilon$ is given, one can determine the size of $R$ by the conditions in \cite{kahn2021nearly}.

\subsection{Review of Kahn--Wright's Construction}\label{sec:review KW}

In this subsection, we review the method in Kahn--Wright~\cite{kahn2021nearly} to construct quasi-Fuchsian surfaces in cusped hyperbolic 3-manifolds, but we will only illustrate the necessary geometric and topological ingredients, without discussing any hard analysis.

\subsubsection{Pants and Hamster Wheels}
Pants are the key building blocks of surfaces in Kahn-Markovi\'c's good pants construction. But in cusped hyperbolic 3-manifolds, mixing does not yield the same equidistribution estimates as in the compact case in regions where the injectivity radius is small. Hence Kahn-Wright introduced hamster wheels as new building blocks to close up the surface. We will briefly review their construction.

For a closed geodesic $\gamma$ in $M$, a \textit{constant turning normal field for $\gamma$} is defined to be a smooth unit normal field $u$ with a constant slope, which is a section of the unit normal bundle for $\gamma$. If the constant slope ${(\theta+2k\pi)}/{b}$ with $l(\gamma) = b+i\theta$ and $k\in \mathbb{Z}$ satisfies $|\theta +2k\pi|<\pi$, then $u$ is called a \textit{slow and constant turning normal field for $\gamma$}. 

A closed geodesic $\gamma$ is called an $(R,\epsilon)$-\emph{good curve} if its complex length $l(\gamma)$ satisfies that $|l(\gamma)-2R|<2\epsilon$.

Let $P$ be a topological pair of pants, which is a thrice-punctured sphere. A continuous map $p\from P\rightarrow M$ is called \textit{a (skew) pair of pants in $M$} up to homotopy if $p_*\from\pi_1(P)\rightarrow \pi_1(M)$ is injective up to conjugacy, and the images of its boundary components $C_1,C_2,C_3$ are homotopic to closed geodesics $\gamma_1,\gamma_2,\gamma_3$ in $M$. In general, we also call $\rho$ a \textit{(free-floating) pair of pants in $\mathrm{PSL}(2,\mathbb{C})$} if $\rho\from\pi_1(P)\rightarrow \mathrm{PSL}(2,\mathbb{C})$ is an injective homomorphism  up to conjugacy.

For each $i$, denote by $a_i$ a simple non-separating arc in $P$ connecting $C_{i-1}$ and $C_{i+1}$. Let $\eta_i$ be the \textit{orthogeodesic} in $M$ homotopic to its image $p(a_i)$, orthogonal to the closed geodesics $\gamma_{i-1}$ and $\gamma_{i+1}$ with constant velocity. Define the half-length $\textbf{hl}(\gamma_i) = d_{\gamma_i}(\eta_{i-1},\eta_{i+1})$ of $\gamma_i$ to be the complex distance between $\eta_{i-1}$ and $\eta_{i+1}$ along $\gamma_i$. The imaginary part of this distance is the oriented angle when we parallel transport the tangent vector of $\eta_{i-1}$ along $\gamma_i$ to the tangent vector of $\eta_{i+1}$. The unit tangent vector to $\eta_j$ on $\gamma_i$ (pointing towards $\gamma_j$) is called a \emph{formal foot} of the pants $P$ on $\gamma_i$. Then each cuff of a pair of pants carries two formal feet and a unique slow and constant turning vector field passes through these two feet on the cuff.

\begin{defi}
    Fix a small number $\epsilon>0$ and a large number $R>0$. An ($R,\epsilon$)\textit{-good pair of pants} in $M$ (or $\PSL(2,\C)$) is a pair of pants $P$ with closed geodesic boundaries satisfying
    \begin{equation}
        |\textbf{hl}(\gamma_i)-R|<\epsilon
    \end{equation}
for all $i$. An \textit{$R$-perfect pair of pants} in $M$ (or $\PSL(2,\C)$) is an ($R,0$)-good pair of pants in $M$ (or $\PSL(2,\C)$).
\end{defi}

The $R$-perfect hamster wheel $H_R$ is defined as the unique planar hyperbolic surface such that
\begin{itemize}
    \item it has $R+2$ geodesic boundaries, each of length $2R$;
    \item it admits a $\Z/R\Z$ action, which fixes two of the boundaries (which are called the outer cuffs) and cyclically permutes the other $R$ boundaries (which are called the inner cuffs).
\end{itemize}
Then $H_R$ is the regular cyclic cover of degree $R$ of the hyperbolic pants with cuff lengths $2,2$ and $2R$.
A continuous map $h\from H_R\rightarrow M$ is called \textit{a hamster wheel in $M$} up to homotopy if $h$ is $\pi_1$-injective and the images of its boundary components are closed geodesics in $M$.

Denote by $\gamma_l$ and $\gamma_r$ the images of two outer cuffs and $\lambda_1,\ldots, \lambda_R$ be the orthogeodesics connecting $\gamma_l$ and $\gamma_r$ in $M$, where the image of each inner cuff lies between two adjacent $\lambda_i$. The foot of $\lambda_i$ at $\gamma_l$ is the unit vector $\eta_{il}$ pointing outward along $\lambda_i$ for each $i$, similarly for $\eta_{ir}$. See Figure \ref{fig1}.
\begin{figure}[ht]
        \centering
        \includegraphics[width=0.5\textwidth]{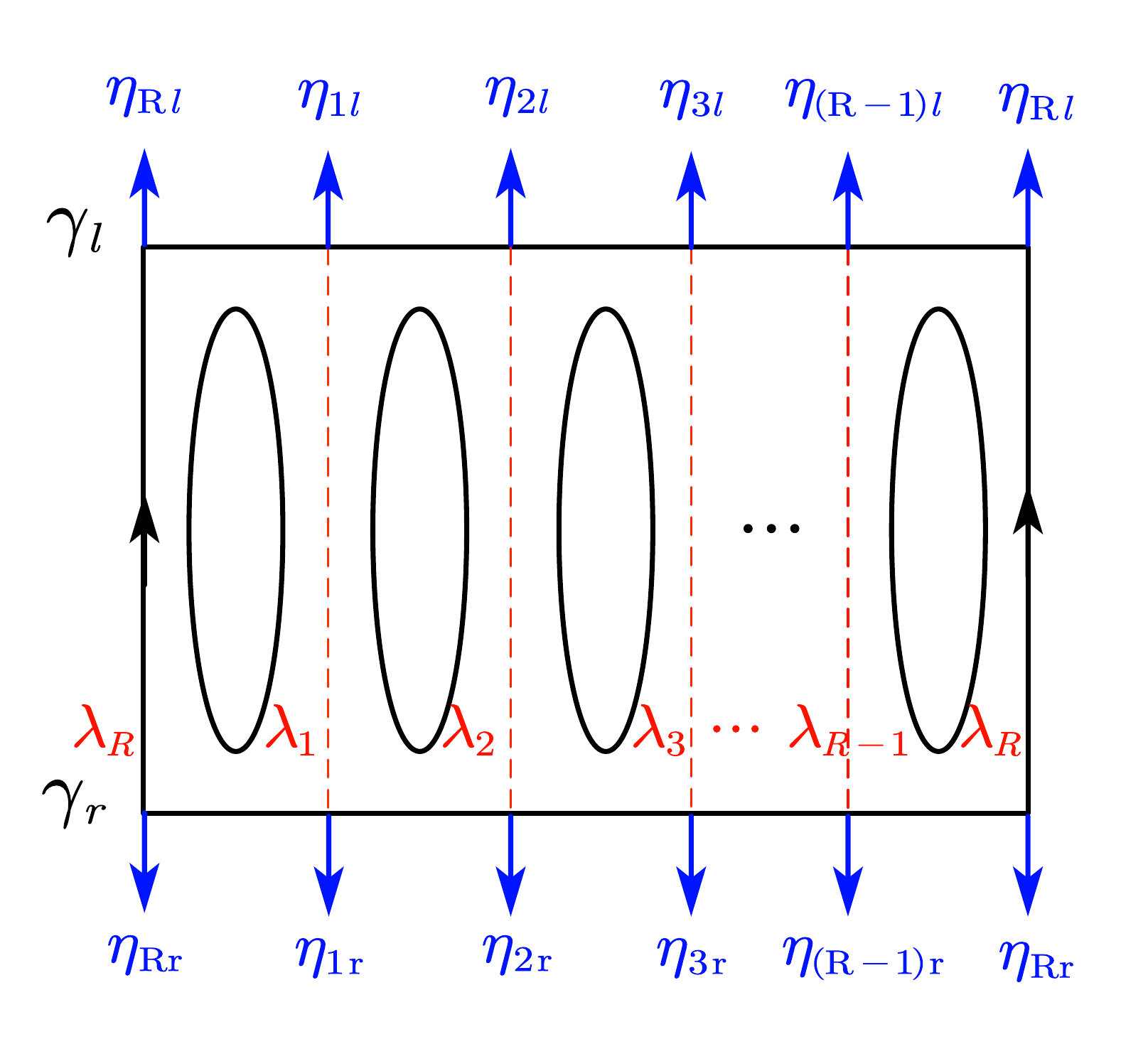}
        \caption{A hamster wheel.}
        \label{fig1}
\end{figure}

\begin{defi}
    Fix a small number $\epsilon>0$ and a large number $R>0$. An ($R,\epsilon$)\textit{-good hamster wheel in $M$} is a hamster wheel $h\from H_R\to M$ with closed geodesic boundaries satisfying the following statement: there exist slow and constant turning vector fields $v_l$ on $\gamma_l$ and $v_r$ on $\gamma_r$ such that for each $i$,
    \begin{enumerate}
        \item The complex distance between $\gamma_l$ and $\gamma_r$ along $\lambda_i$ satisfies 
        $$
        |d_{\lambda_i}(\gamma_l,\gamma_r) - (R-2\log\sinh1)|<\frac{\epsilon}{R};
        $$
        \item The complex distance between the feet $\eta_{il}$ and $\eta_{(i+1)l}$ along the outer boundary geodesics $\gamma_l$ satisfies 
        $$
        |d_{\gamma_l}(\eta_{il},\eta_{(i+1)l})-2|<\frac{\epsilon}{R}.
        $$ 
        Similarly, the same estimate holds with $l$ replaced by $r$.
        \item The angle between the feet $\eta_{il}$ and the vector field satisfies 
        $$
        |\eta_{il}-v_l|<\frac{\epsilon}{R},
        $$ similarly if we replace $l$ by $r$.
    \end{enumerate}
\end{defi}

For each inner cuff $\gamma$ of a good hamster wheel $h\from H_R\to M$, two formal feet are chosen as follows. Consider the unit normal vector on $\gamma$ that is the unit tangent vector to the orthogeodesic between $\gamma$ and an adjacent inner cuff and pointing towards the adjacent inner cuff: there are two such vectors, denoted by $\al_+$ and $\al_-$. One can verify that $|\al_+-\al_--R|$ is small; hence we can pick a unique $\al'_+$ and $\al'_-$ such that $\al'_+-\al_+=\al_--\al'_-$ are small and $\al'_+-\al'_-=\hl(\gamma)$. Then $\al'_+,\al'_-$ are called the formal feet on $\gamma$, and there is a unique slow and constant turning vector field through $\al'_+$ and $\al'_-$. The half-lengths of outer and inner cuffs are then well-defined using the slow and constant turning vector fields.

For simplicity, an $(R,\epsilon)$-good pair of pants or an $(R,\epsilon)$-good hamster wheel is called an \textit{$(R,\epsilon)$-good component}. A pants or hamster wheel is called $R$-perfect, if it is $(R,0)$-good. The cuffs of an $(R,\epsilon)$-good component are always $(R,\epsilon)$-good curves.

We note that for a cuff (or inner cuff) $\gamma$ of a pants (or a hamster wheel) $P$, the two formal feet are symmetric in the unit normal bundle $N^1(\gamma)$; hence we just denote a formal foot of $P$ on $\gamma$ by $\foot_\gamma(P)$, without indicating which formal foot if there is no ambiguity.

\subsubsection{Well-matched Components and QF Surfaces}

The formal feet of cuffs of good pants and the inner cuffs of good hamster wheels have been defined, and they will be used to detect the twist parameter when matching two good components along a common cuff.

\begin{defi}\label{def:well-matched}
    Two good components (pants or hamster wheels) $P$ and $Q$ with a common geodesic $\gamma$ as one of their boundaries are $(R,\epsilon)$-\textit{well-matched at $\gamma$} if either of the following holds: 
    \begin{enumerate}
        \item For each formal foot $\al$ of $P$ on $\gamma$, there is a formal foot $\beta$ of $Q$ on $\gamma$ such that        
        \begin{equation}
            |\al-\beta-(1+i\pi)|<\frac{\epsilon}{R};
        \end{equation}
        \item Otherwise, the constant turning vector fields form a bend angle of at most $100\epsilon$.
    \end{enumerate} 
    We define $s_\gamma=\al-\beta$ as the \emph{shear} along $\gamma$ for the first case, and as the bending angle for the second case. We also say that two good components are glued through an $(R,\epsilon)$\emph{-good matching}.
\end{defi}

Recall that a collection of quasi-Fuchsian surface subgroups is called \emph{ubiquitous}, if for any two hyperbolic planes $\Pi,\Pi'\subset \H^3$ with positive distance between them, there exists a surface subgroup whose boundary circle lies between $\partial\Pi$ and $\partial\Pi'$. Kahn-Wright showed that for any given $\epsilon>0$, when $R$ is sufficiently large,
\begin{itemize}
    \item all closed surface representations in $\PSL(2,\C)$ that are assembled by $(R,\epsilon)$-good components via $(R,\epsilon)$-good matching are nearly Fuchsian [\citenum{kahn2021nearly}, Theorem 2.2]; 
    \item the collection of such surfaces in $M$ is ubiquitous [\citenum{kahn2021nearly}, Theorem 1.1].
\end{itemize}
These two statements are essentially all we need from \cite{kahn2021nearly} to demonstrate the lower bound.

\subsection{Good QF Representations and Bending Deformations}\label{sec:bending}

In Section 3 of \cite{kahn2012counting}, Kahn-Markovi\'c showed that given any QF representation of a closed surface parametrized with a pants decomposition, one can deform the bending angle (i.e., the imaginary part of the shear parameter) along some cuffs up to $3\pi/4$, and obtain another QF representation whenever the collars around these cuffs are disjoint. Here the width of collars is chosen to be a universal constant. In our setting, we need a similar result for surfaces made of good components. 

Suppose $S$ is a topological surface (with or without boundary). A \emph{component decomposition} is a collection $\mathcal{C}$ of disjoint simple closed curves such that every connected component of $S\backslash\mathcal{C}$ is a component (i.e., a pants or a hamster wheel).

A \emph{generalized component} is essentially a finite cover of a component. When taking finite covers of a surface with a component decomposition, the lifts of each component are generalized components; then we get a \emph{generalized component decomposition} of the covers.

\begin{defi}
    A pair $(\tilde{P},\chi)$ is a generalized component if $\tilde{P}$ is a compact surface with boundary and $\chi$ is a finite covering map $\chi\from\tilde{P}\to P$, where $P$ is a component. We denote by $\chi_*$ the induced homomorphism $\chi_*\from\pi_1(\tilde{P})\to\pi_1(P)$, and we sometimes call $\tilde{P}$ a generalized component if $\chi$ is understood.

    Let $(\tilde{P},\chi)$ be a generalized component and $\tilde{\rho}\from\pi_1(\tilde{P})\to\PSL(2,\C)$ a representation. We say that $\tilde{\rho}$ is admissible with respect to $\chi$, if it factors through $\chi_*$, i.e., there exists $\rho\from\pi_1(P)\to\PSL(2,\C)$ such that $\tilde{\rho}=\rho\circ\chi_*$. 
\end{defi}

Let $\tilde{C}_j$ ($j=1,2,\dots,k$) denote the cuffs (the boundary curves) of $\tilde{P}$, and let $C_i$ ($i=1,2,\dots,s$) denote the cuffs of $P$. Then $\chi$ maps each $\tilde{C}_j$ onto some $C_i$ with some degree $m_j\in\Z_+$; we say that such $\tilde{C}_j$ is a degree-$m_j$ curve. For every admissible $\tilde{\rho}$,
\begin{itemize}
    \item the slow and constant turning vector field on $C_i$ uniquely determines a slow and constant turning vector on $\tilde{C}_j$;
    \item when $C_i$ is a cuff of a pair of pants or an inner cuff of a hamster wheel, we define the half-length of $\tilde{C}_j$ as $\hl(\tilde{C}_j)=\hl(C_i)$. One notes that $\hl(\tilde{C}_j)$ is then not the half of the length of $\tilde{C}_j$.
\end{itemize}

Let $S$ be an oriented closed topological surface. We say a collection $\mathcal{C}$ of disjoint simple closed curves is a \emph{generalized component decomposition}, if for every connected component $\tilde{P}$ of $S\backslash\mathcal{C}$, there is an associated finite cover $\chi\from\tilde{P}\to P$. Let $f\from\pi_1(S)\to\PSL(2,\C)$ be a representation. We may assume that $f$ also satisfies that:
\begin{itemize}
    \item For each generalized component $\tilde{P}$ of the component decomposition of $S$, the restriction $f|\from \pi_1(\tilde{P})\to\PSL(2,\C)$ is admissible.
    \item Given a curve $C\in\mathcal{C}$, there are two generalized components $\tilde{P}_1$, $\tilde{P}_2$ having $C$ as a cuff that lie on opposite sides of $C$. We assume that the restriction of each finite cover $\chi_i\from\tilde{P}_i\to P_i$ over the curve $C$ has the same degree, for $i=1,2$. If the half-lengths of $C$ are well-defined from both $\pi_1(\tilde{P}_i)\to\PSL(2,\C)$, then they are exactly the same.
\end{itemize}
Moreover, let $C_i\in P_i$ be the cuff so that $\chi(C)=C_i$ and $f_i\from\pi_1(P_i)\to\PSL(2,\C)$ the representation such that $f|_{\pi_1(\tilde{P}_i)}=f_i\circ (\chi_i)_*$. When the half-lengths of the $C_i$ are well-defined (and then equal), we define the \emph{shear} along $C$ to be the shear along $f_1(C_1)=f_2(C_2)$ between $P_1$ and $P_2$.

We say a generalized component representation $\tilde{\rho}\from\pi_1(\tilde{P})\to\PSL(2,\C)$ is $(R,\epsilon)$-\emph{good}, if it is admissible with respect to some finite cover $\chi\from\tilde{P}\to P$ for some $(R,\epsilon)$-good component $\rho\from \pi_1(P)\to\PSL(2,\C)$. We can then define that two generalized good components are $(R,\epsilon)$-well-matched, similar to Definition \ref{def:well-matched}. Furthermore, we say that a representation $f\from\pi_1(S)\to\PSL(2,\C)$ is $(R,\epsilon)$-good, if $S$ has a generalized component decomposition $\mathcal{C}$ such that
\begin{itemize}
    \item each connected component of $S\backslash\mathcal{C}$ is a generalized $(R,\epsilon)$-good component;
    \item along each curve $C\in\mathcal{C}$, the two associated generalized components are $(R,\epsilon)$-well-matched along $C$.
\end{itemize}

Suppose $f\from \pi_1(S)\to\PSL(2,\C)$ is an $(R,\epsilon)$-good representation. Let ${f}_0\from \pi_1(S)\to \PSL(2,\R)$ be the surface representation by replacing the good generalized components with $R$-perfect generalized component and gluing them in the same topological pattern via $R$-perfect matching (i.e., $(R,0)$-good matching). Then by Kahn--Wright's result [\citenum{kahn2021nearly}, Theorem 2.2] for any $K>1$, there exist $\epsilon>0$ and $R_0>0$ such that when $R>R_0$, $f(\pi_1(S))$ is $K$-quasiconformally conjugate to ${f}_0(\pi_1(S))$. Now we suppose that $\mathcal{C}_0$ is a collection of cuffs on $S$, and let $\mathcal{K}(f_0,\mathcal{C}_0)$ be the largest number such that the collection of collars of width $\mathcal{K}(f_0,\mathcal{C}_0)$ around the curves from $\mathcal{C}_0$ is disjoint on the hyperbolic surface $\tilde{S}=\H^2/{f}_0(\pi_1(S))$. For each curve $C\in\mathcal{C}_0$, we alter the associated bending angle along $C$ on $f(\pi_1(S))$ by $\pi/2$ or $-\pi/2$, and then obtain a new representation $f_2\from \pi_1(S)\to\PSL(2,\C)$.

With the above settings, we state the following result.

\begin{theorem}\label{thm:bending}
    There exist $D,\epsilon,R_0>0$ such that when $R>R_0$:
    if $\mathcal{K}(f_0,\mathcal{C}_0)\geq D$, then $f_2\from \pi_1(S)\to\PSL(2,\C)$ is $K_1$-QF, where $K_1$ depends only on $D$ and $\epsilon$.
\end{theorem}

\begin{remark}
    The above theorem is analogous to Theorem 3.1 in \cite{kahn2012counting} for good components, and the proof will be the same, so we skip it.
\end{remark}

\subsection{Assembling good representations}\label{sec:amalgamating}

In this subsection, we cut two closed surfaces that share a common closed geodesic and re-glue them into one surface. We show that the new surface is maximal, i.e., it is not a nontrivial covering space of another surface, when the initial two surfaces satisfy some certain condition.

We first recall the definition of \emph{height} from \cite{kahn2021nearly}. Let $\epsilon_0$ be a universal constant, independent of $M$, so that the horoballs about cusps consisting of those points with injectivity radius at most $\epsilon_0$ are disjoint (see \cite[Chapter D]{benedetti1992lectures}). The \emph{height} of a point in $M$ is defined as the signed distance to the boundary of these horoballs, so that the height is positive if the point is in one of the horoballs and negative otherwise.

Kahn-Wright~\cite{kahn2021nearly} showed that, for a given $\epsilon>0$ and $R$ sufficiently large, one can find many $(R,\epsilon)$-good representations in $\Gamma$. Moreover, suppose $\gamma$ is an $(R,\epsilon)$-good curve with height at most $h_C=50\log R$ whose associated hyperbolic element $A_\gamma\in\Gamma$ is primitive, i.e., $A_\gamma \neq \alpha^n$ for some $\alpha\in\Gamma$ and $|n|>1$. 
It follows from \cite[Section 3,4,5]{kahn2021nearly} (the statements about the equidistribution of good components in $M$ and the constructions of QF surfaces) that one can find two $(R,\epsilon)$-good representations $f_i\from \pi_1(S_i)\to\Gamma$, $i=1,2$, where the $S_i$ are two closed surfaces with component decomposition $\mathcal{C}_i$ (not just a generalized component decomposition), and components $P_i^{\pm}$ with the following properties: 
\begin{enumerate}
    \item There are two oriented curves $C_i\in\mathcal{C}_i$, and $c_i\in\pi_1(S_i)$ in the conjugacy class of $C_i$ such that $f_1(C_1)=f_2(C_2)$ is the conjugacy class of $A_\gamma$.

    \item The good components $P_i^\pm$ are in $S_i$ such that $\gamma$ is a positively oriented boundary component of $P_i^+$ and negatively oriented for $P_i^-$. The formal feet or slow and constant turning vector fields of $P_i^\pm$ on $\gamma$ should satisfy:
    \begin{itemize}
        \item If $P_2^+$ and $P_1^-$ both have formal feet on $\gamma$, then 
        \begin{equation}
            |\foot_\gamma(P_2^+)-\foot_\gamma(P_1^-)-\frac{\pi}{2}i|\leq\frac{\epsilon}{R};
        \end{equation}
        otherwise the bending angle between the slow and constant turning vector fields of $P_2^+$ and $P_1^-$ lies in $[{\pi}/{2}-100\epsilon, {\pi}/{2}+100\epsilon]$. Here $100\epsilon$ comes from the well-matchedness along the cuffs without formal feet. 
        \item The same condition for $P_1^+$ and $P_2^-$.
    \end{itemize}
\end{enumerate}
After replacing the $S_i$ by appropriate finite covers, we may assume the following also hold:
\begin{enumerate}[resume]
    \item The curve $C_i$ is a non-separating simple closed curve in $S_i$, $i=1,2$.

    \item  The surfaces $S_1$ and $S_2$ have the same genus.

    \item Let $\rho_i^0\from \pi_1(S_i)\to\PSL(2,\R)$ be the Fuchsian representation obtained as in Section \ref{sec:bending}. We may assume that $\mathcal{K}(\rho_i^0,\{C_i\})>D$, where $D$ is the constant from Theorem \ref{thm:bending}.
\end{enumerate}
We note that after taking finite covers, we will only assume that $\mathcal{C}_i$ is a generalized component decomposition of $S_i$ while the $f_i$ are still $(R,\epsilon)$-good representations. Now we fix such good representations $f_i$, surfaces $S_i$, oriented curves $C_i$, for $i=1,2$, and the primitive element $A_\gamma$.

Suppose $\tilde{S}_1$ and $\tilde{S}_2$ are two degree-$n$ covers of $S_1$ and $S_2$, respectively, such that 
\begin{itemize}
    \item there is no intermediate cover between $\tilde{S}_1$ and $S_1$, $\tilde{S}_2$ and $S_2$ (i.e., they are primitive covers);
    \item for some integer $1\leq k\leq n-1$, there are two lifts $\tilde{C}_1\subset \tilde{S}_1$ and $\tilde{C}_2\subset \tilde{S}_2$ of the curves $C_1$ and $C_2$, respectively, of degree $k$  ($k<n$ is required in the proof of Theorem \ref{thm16}).
\end{itemize} 
Denote the induced representations by $\tilde{f}_i\from\pi_1(\tilde{S}_i)\rightarrow \Gamma$, $i=1,2$. 
Then the $\tilde{f}_i$ are $(R,\epsilon)$-good representations that also satisfy the above five conditions, except that
\begin{equation}
    \tilde{f}_1(\pi_1(\tilde{S}_1))\cap\tilde{f}_2(\pi_1(\tilde{S}_2))
    =\langle A_\gamma^k\rangle.
\end{equation}
Here we can take the base point $*_i$ of $\tilde{S}_i$ on $\tilde{C}_i$, and $*$ of $M$ on $\gamma$. Since the $\tilde{f}_i$ are both nearly quasi-Fuchsian representations and their associated subsurfaces in $M$ meet at $\gamma$ with an angle close to $\pi/2$, their universal covers in $\H^3$ are two nearly geodesic planes meeting at a geodesic with an angle close to $\pi/2$. Thus we can assume $\tilde{f}_1(\pi_1(\tilde{S}_1))\cap\tilde{f}_2(\pi_1(\tilde{S}_2))
    =\langle A_\gamma^k\rangle.$

Now we amalgamate the $\tilde{f}_i$ as follows. 
Cut $\tilde{S}_i$ along $\tilde{C}_i$ to get a surface $\tilde{S}'_i$ with two boundary components $\tilde{C}^+_i$ and $\tilde{C}^-_i$.
We glue the surfaces $\tilde{S}'_1$ and $\tilde{S}'_2$ by gluing $\tilde{C}^\pm_1$ to $\tilde{C}^\pm_2$, and obtain a closed surface $\hat{S}$. The surface $\hat{S}$ is well-defined up to a twist by $l(\gamma)$ with period $k$. We denote by $\hat{C}^\pm$ the curve on $\hat{S}$ that was produced by gluing $\tilde{C}^\pm_1$ to $\tilde{C}^\pm_2$. Then $\hat{S}$ has a natural generalized component decomposition $\hat{\mathcal{C}}$. We let $\hat{\mathcal{C}}_0=\{\hat{C}^+,\hat{C}^-\}$.

The above operation induces a representation $\hat{f}\from\pi_1(\hat{S})\to\Gamma$, and we claim that it is a QF representation. We orient the curves $\hat{C}^\pm$ such that for any $c\in\pi_1(\hat{S})$, whenever $c$ is in the conjugacy class of $\hat{C}^\pm$, we have $\hat{f}(c)$ lie in the conjugacy class of $A^k_\gamma$ in $\Gamma$. Now we consider a representation $\hat{f}_1\from\pi_1(\hat{S})\to\PSL(2,\C)$ obtained by changing the bending angle along $\hat{C}^\pm$ such that $\hat{f}_1$ is an $(R,\epsilon)$-good representation. Then by Theorem \ref{thm:bending} and Condition 5 for the $\tilde{S}_i$, we know that $\hat{f}(\pi_1(\hat{S}))$ is a QF subgroup of $\Gamma$. 

Finally, $\hat{f}(\pi_1(\hat{S}))<\Gamma$ is maximal, via the same argument as \cite[Theorem 4.2]{kahn2012counting}. We restate the result as the following theorem.

\begin{theorem}\label{thm16}
    Let $\hat{f}\from\pi_1(\hat{S})\rightarrow\Gamma$ be the quasi-Fuchsian representation as above. Then $\hat{G}=\hat{f}(\pi_1(\hat{S}))$ is a maximal QF surface subgroup, that is, there is no intermediate QF surface subgroup between $\Gamma$ and $\hat{G}$.
\end{theorem}

\subsection{Counting Maximal QF Surface Subgroups}

In this subsection, we will count the number of maximal QF surface subgroups we can produce as in Section \ref{sec:amalgamating} and then prove the lower bound in Theorem \ref{thm1}.

Suppose that $g_0\geq 2$ is the genus of the closed surfaces $S_1$ and $S_2$. Then the Euler characteristic of covers $\tilde{S}_1$ and $\tilde{S}_2$ of degree-$n$ is $$\chi(\tilde{S}_1)=\chi(\tilde{S}_2)=n\chi(S_1)=n\chi(S_2) = n(2-2g_0).$$
After cutting $\tilde{S}_1$ and $\tilde{S}_2$ along the non-separating curves $\tilde{C}_1$ and $\tilde{C}_2$, the Euler characteristic does not change. Gluing the two surfaces together to get $\hat{S}$ of genus $\hat{g}$, the Euler characteristic satisfies
$$\chi(\hat{S})=2-2\hat{g}=\chi(\tilde{S}_1)+\chi(\tilde{S}_2)=2n(2-2g_0),$$
\begin{equation}\label{eq14}
    \hat{g}=n(2g_0-2)+1,\qquad n=\frac{\hat{g}-1}{2g_0-2}.
\end{equation}

In \cite[Section 4.4]{muller2002character}, it is shown that, for the surface group $\Gamma_0$ of genus $g_0$, the number of maximal subgroups in $\Gamma_0$ of index $n$ satisfies that 
\begin{equation}\label{eq15}
    m_n(\Gamma_0) = 2n(n!)^{2g_0-2}(1+o(1)),
\end{equation}
for sufficiently large $n$. Let $\Gamma_0=\pi(S_0)$ and $C_0$ be a non-separating simple closed curve on $S_0$. For $1\leq k\leq n$, let $m_n(\Gamma_0,C_0,k)$ be the number of maximal degree-$n$ covers of $S_0$ such that the curve $C_0$ has a lift of degree $k$. Then $m_n(\Gamma_0,C_0,k)$ does not depend on the choice of $C_0$, so we write it as $m_n(\Gamma_0,k)$. We recall the following theorem from \cite{kahn2012counting}.

\begin{theorem}[{\cite[Theorem 4.3]{kahn2012counting}}]\label{thm23}
    For some $1\leq k\leq n-1$, $k=k(n,g_0)$, we have
    \begin{equation}\label{eq:m_n}
        m_n(\Gamma_0,k)\geq(n!)^{2g_0-2}.
    \end{equation}
\end{theorem}

Now we state the lower bound as follows.

\begin{theorem}\label{thm24}
    Let $M=\mathbb{H}^3/\Gamma$ be a cusped hyperbolic $3$-manifold. For sufficiently large $g$, the number of commensurability classes of quasi-Fuchsian surface subgroups in $\Gamma$ of genus at most $g$ satisfies 
    \begin{equation}
        n(g,M)\geq(C_2g)^{2g},
    \end{equation} for some constant $C_2>0$ depending only on $\Gamma$. 
\end{theorem}
\begin{proof}
    By \eqref{eq14} and \eqref{eq:m_n}, for sufficiently large $n$, we choose $1\leq k\leq n-1$ such that for $i=1,2$, $m_n(\pi_1(S_i),k)\geq (n!)^{2g_0-2}$. We then amalgamate any two maximal covers as in Section \ref{sec:amalgamating} and produce a maximal surface subgroup of genus $\hat{g}$. Then the number of maximal surface subgroups of genus $\hat{g}$ is bounded below by
    \begin{equation}
        \begin{aligned}
            m_n(\pi_1(S_1),k)\cdot m_n(\pi_1(S_2),k)&\geq ((n!)^{2g_0-2})^2
            >\left(\sqrt{n}\left(\frac{n}{e}\right)^n\right)^{4g_0-4}\\
            &= n ^ {2g_0-2}\left(\frac{n}{e}\right)^{2\hat{g}-2}
            = \left(\frac{\hat{g}-1}{2g_0-2}\right) ^ {2g_0-2}\left(\frac{\hat{g}-1}{e(2g_0-2)}\right)^{2\hat{g}-2}\\
            &\geq(\hat{C}_2\hat{g})^{2\hat{g}+2g_0-4},
        \end{aligned}
    \end{equation}
    for some constant $\hat{C}_2>0$. Therefore for any $g$ sufficiently large, let $n$ be the positive integer such that $$n(2g_0-2)+1\leq g<(n+1)(2g_0-2)+1.$$
    Let $\hat{g}=n(2g_0-2)+1$, then we have
    \begin{equation}
        n(g,M)\geq(\hat{C}_2\hat{g})^{2\hat{g}+2g_0-4}\geq(C_2 g)^{2g},
    \end{equation}
    for some constant $C_2>0$.
\end{proof}

\begin{remark}
    In \cite[Section 4.2]{kahn2012counting}, there is a typo in the calculation of $g_n$. See \eqref{eq14} for the correct genus $g_n$ (here denoted by $\hat{g}$).
    In \cite[Theorem 4.3]{kahn2012counting}, there are some typos in the exponent of $n!$ and $(n-1)!$, and a missing term in the estimate for $m_n(\Gamma_0)$. See \eqref{eq15} and \eqref{eq:m_n} for the correct estimates for $m_n(\Gamma_0)$ and $m_n(\Gamma_0,k)$.
\end{remark}

Therefore, Theorem \ref{thm1} follows from Theorem \ref{thm8} and Theorem \ref{thm24}.

%% file: pA.tex
\section{Counting Purely Pseudo-Anosov Surface Subgroups}\label{sec:MCG}

In this section, we present a lower bound for the number of commensurability classes of surface subgroups in mapping class groups, by combining our lower bound from Theorem \ref{thm24} with Kent--Leininger's construction of surface subgroups in mapping class groups.

Let $M_8$ be the figure-eight knot complement, $S_{h,0}$ the closed surface of genus $h$ for $h\geq 4$, and $S_{1,3}$ the thrice-punctured torus. In \cite{kent2024atoroidal}, Kent and Leininger construct a type-preserving homomorphism as follows.
\begin{theorem}[\citenum{kent2024atoroidal}, Theorem 2]
    There is a type-preserving representation $\Delta$ from the fundamental group $\pi_1(M_{8})$ of the figure-eight knot complement to the mapping class group $\mathrm{Mod}(S_{1,3})$.
\end{theorem}
A \textit{type-preserving} representation is a homomorphism that takes peripheral (i.e., parabolic) elements to reducible mapping classes and hyperbolic elements to pseudo-Anosov mapping classes. There are infinitely many commensurability classes of totally geodesic closed immersed surfaces $S_i$ in the figure-eight knot complement \cite{maclachlan1986fuchsian}. Kent and Leininger [\citenum{kent2024atoroidal}, Corollary 19] give infinitely many commensurability classes of purely pseudo-Anosov closed surface subgroups of $\mathrm{Mod}(S_{1,3})$ by mapping $\pi_1(S_i)$ injectively into $\mathrm{Mod}(S_{1,3})$. 
Therefore, by passing to some finite-index subgroups of $\mathrm{Mod}(S_{1,3})$, the number of commensurability classes of purely pseudo-Anosov closed surface subgroups of $\mathrm{Mod}(S_{h,0})$ is also infinite.

Using the growth rate of maximal Fuchsian subgroups in $\pi_1(M_9)$ in \cite{vulakh1991classification}, they obtain the following. 

\begin{theorem}[\citenum{kent2024atoroidal}, Theorem 6]\label{thm26}
    The number of commensurability classes of purely pseudo-Anosov subgroups of $\mathrm{Mod}(S_{h,0})$ for all $h\geq 4$ and $\mathrm{Mod}(S_{1,3})$ that are isomorphic to the fundamental group of a surface of genus at most $g$ is bounded below by a strictly increasing linear function of $g$.
\end{theorem}

In the proof of Theorem \ref{thm26}, there is a universal constant $m$ such that the natural map from the set of conjugacy classes of convex cocompact surface subgroups of $\pi_1(M_8)$ to the set of conjugacy classes of purely pseudo-Anosov subgroups of $\mathrm{Mod}(S_{1,3})$ induced by $\Delta$ is $m$-to-$1$. By passing to finite-index subgroups of $\pi_1(M_8)$ so that the image of $\Delta$ is a subgroup of $\mathrm{Mod}(S_{h,0})$, the new comparison constant $m=m(h)$ depends only on the genus $h$.

\begin{proof}[Proof of Corollary \ref{cor3}]
    Let $p(g)$ be the number of commensurability classes of purely pseudo-Anosov closed surface subgroups of genus at most $g$ in $\mathrm{Mod}(S_{h,0})$, and let $c(g)$ be the number of commensurability classes of convex cocompact closed surface subgroups of genus at most $g$ of $\pi_1(M_8)$.
    By the above argument, we have
\begin{equation}
    p(g)\geq \frac{c(g)}{m}.
\end{equation}

Since a quasi-Fuchsian surface subgroup of $\pi_1(M_8)$ is convex cocompact, these immersed closed surfaces must be geometrically finite. Thus, $c(g)$ is the number $n(M_8,g)$ of quasi-Fuchsian surface subgroups of $\pi_1(M_8)$ of genus at most $g$ up to commensurability in Theorem \ref{thm1}, and $m$ can be absorbed into the constant $C_2$. Hence Corollary \ref{cor3} follows.
\end{proof}

%% file: Coannular.tex
\section{Coannular Surfaces}\label{sec:coannular}

In this section, we prove \Cref{thm:coannular} by explicitly constructing infinitely many coannular essential surfaces of some fixed genus, up to homotopy. Let $M=\mathbb{H}^3/\Gamma$ be a cusped hyperbolic $3$-manifold.
\begin{defi}[Coannular surface]
    A $\pi_1$-injective closed surface $f\from S\to M_0$ of genus  at least $2$ is \textit{coannular to a toroidal boundary component} $T$ in a compact $3$-manifold $M_0$ with torus boundary components if there is a $\pi_1$-injective annulus $A$ in $M_0$ with a boundary component on each of $S$ and $T$. The slope of $A\cap T$ is called a \textit{coannular slope} of $S$ on $T$.
\end{defi}

Initially, we follow the strategy in \cite{clrEssentialClosedSurfaceCuspedM} by passing $M$ to a finite cover $\widehat{M}  $ with at least 3 boundary components, and find a connected, non-separating properly embedded incompressible surface $F$ in $\widehat{M}  $ which is disjoint from some torus boundary component of $\widehat{M}  $. We first recall two lemmas from \cite{clrEssentialClosedSurfaceCuspedM}.

\begin{lem}[\citenum{clrEssentialClosedSurfaceCuspedM}, Lemma 2.1]\label{lem:clrlemma2.1}
    A cusped hyperbolic $3$-manifold $M$ has a finite cover $\widehat{M}  $ with at least $3$ torus boundary components. 
\end{lem}

\begin{lemma}[\citenum{clrEssentialClosedSurfaceCuspedM}, Lemma 2.2]\label{lem:clrlemma2.2}
    Suppose $\widehat{M}  $ is a cusped orientable $3$-manifold with at least $3$ torus boundary components. Then there is a connected, non-separating, properly embedded, incompressible surface $F$ in $\widehat{M}  $ which is disjoint from
    at least one torus boundary component. Furthermore, $[ \partial F] \neq 0$ in $ H_1 ( \partial \widehat{M}   )$
    and meets the boundary efficiently. 
\end{lemma}
Since $M$ is aspherical, atoroidal, and acylindrical, a \poi surface $F\hookrightarrow M$ with boundary (which is not an annulus) must have Euler characteristic at most $-1$. 
Cooper--Long--Reid~[\citenum{clrEssentialClosedSurfaceCuspedM}, Theorem 2.3] first constructs a closed coannular surface $S'$ in an infinite cyclic cover $M_F$, by tubing the boundary components of two copies of $F$ in $M_F$. All tubes have the \emph{same} height (i.e., their projections in $\widehat{M} $ wrap the boundary tori the same number of times). They then construct an immersed essential surface $S$ in $M$ by maximally compressing $S'$. There may exist a compressing disk annihilating the desired annuli; see \Cref{fig:compressing}.
    \begin{figure}[htbp]
        \centering
        \includegraphics[width=0.35\linewidth]{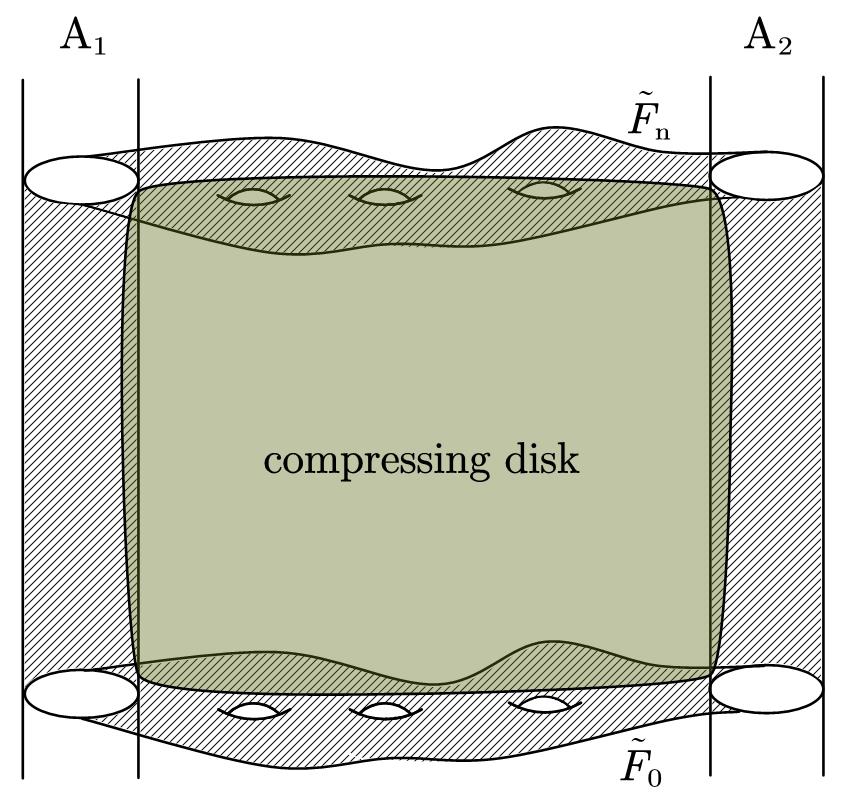}
        \caption{A possible compressing disk in $M_F$.}
        \label{fig:compressing}
    \end{figure}

Instead, we directly construct an immersed essential coannular closed surface in a finite cover by tubing 
with \emph{possibly distinct} heights and then applying Maskit's Kleinian combination theorems. The original paper of Maskit is \cite{maskit1993klein}, and some applications are presented in \cite[Chapter 1]{krushkal_1986kleinian}. Before stating and proving the main result, we recall Maskit's two Kleinian combination theorems.

\begin{defi}[Precisely invariant under a subgroup and doubly cusped]
    For a subgroup $J$ of a Kleinian group $G$, a topological disk $B$ in the Riemann sphere $\hat{\mathbb{C}}$ is called \emph{precisely invariant under $J$} if $g(B)=B$ for all $g \in J$, and $g(B)\cap B =\varnothing$ for all $g\in G\setminus J$. 
    
    The fixed point of a rank-one parabolic subgroup $J$ of $G$ is called \emph{doubly cusped} if there are two disjoint circular disks $B_1, B_2\in \hat{\C}$ such that $B_1 \cup B_2 $ is precisely invariant under $J$ in $G$.
\end{defi}

\begin{theorem}[Maskit's first Kleinian combination theorem, \cite{maskit1993klein}]\label{1st KCT}
    Suppose $G_1$ and $G_2$ are two geometrically finite Kleinian groups with a common subgroup $J$, where the index of $J$ is at least 2 in both groups. If $J$ is an infinite cyclic group generated by a parabolic element, and there is a Jordan curve that bounds two open disks $B_1$ and $B_2$ in $\hat{\mathbb{C}}$ such that $B_i$ is precisely invariant under $J$ in $G_i$ for $i=1,2$, then $G=\langle G_1,G_2\rangle\cong G_1\ast_{J}G_2$ is a geometrically finite Kleinian group.
\end{theorem}

\begin{theorem}[Maskit's second Kleinian combination theorem, \cite{maskit1993klein}]\label{2nd KCT}
    Suppose $Q$ is a geometrically finite Kleinian group with two subgroups $J_1$ and $J_2$. Let $B_1$ and $B_2$ be two open disks in $\hat{\mathbb{C}}$ with $\hat{\mathbb{C}}\setminus(\overline{B}_1\cup \overline{B}_2) \ne \varnothing$ such that $g(B_1)\cap B_2=\varnothing$ for all $g\in Q$, and $B_i$ is precisely invariant under $J_i$ in $Q$, $i=1,2$. If $f\in \PSL(2,\mathbb{C})\setminus Q$ maps $\hat{\mathbb{C}}\setminus B_1$ onto the interior of $B_2$, satisfying $J_2=fJ_1f^{-1}$ and $f(\partial B_1)=\partial B_2$, then $G=\langle Q,f \rangle \cong 
    Q\ast_{f}$ is the HNN-extension of $Q$ by $f$ and a geometrically finite Kleinian group.
\end{theorem}

Now we present our construction of a coannular closed surface. Although some part of the proof that involves Maskit's second Kleinian combination theorem can be viewed as a corollary of 
\cite[Theorem 8.8]{Baker-Cooper08},  we directly prove the desired result for completenes. For a \poi closed surface $f\from S\to M$, it is called \emph{disjoint from a torus boundary component $T$ of $M$} if $f_*(\pi_1(S))\cap \pi_1(T)$ is trivial.

\begin{theorem}\label{thm6.5}
Let $\widehat{M}  $ be a cusped hyperbolic $3$-manifold with at least 3 torus boundary components. Then $\widehat{M}  $ contains an immersed essential coannular closed surface of genus $g\geq 2$ that is disjoint from at least one torus boundary component of $\widehat{M}$.

\end{theorem}
\begin{proof}
    By \Cref{lem:clrlemma2.2}, $\widehat{M}  $ has a properly embedded  $2$-sided non-separating incompressible surface $F$ that meets the boundary of $\widehat{M}  $ efficiently and is disjoint from some torus boundary component of $\widehat{M}  $, with $[ \partial F ] \neq 0$ in $ H_1 ( \partial \widehat{M}   )$. Suppose $F$ has $m\geq1$ boundary components, all of which are contained in $\partial\widehat{M}  $.

    Suppose the boundary curves of $F$ are $\gamma_0,\gamma_1,\dots,\gamma_{m-1}$, where $\gamma_j$ is contained in a boundary torus $T_j$ of $\widehat{M}  $. Here the $T_j$ are not necessarily distinct. We fix a basepoint $x\in\gamma_0$ for all fundamental groups that appear later. Let $\langle \alpha_j,\beta_j \rangle<\PSL(2,\C)$ be the rank-2 parabolic subgroup of $\pi_1(\widehat{M}  )$ associated to $T_j$, in which we assume that $\langle\beta_j\rangle$ does not contain the element represented by $\gamma_j$, for all $j$. Let $H$ be the image of $\pi_1(F,x)$ under the inclusion $\pi_1(F,x)\hookrightarrow \pi_1(\widehat{M}  ,x)\hookrightarrow \PSL(2,\mathbb{C})$. Then $H$ is finitely generated since $F$ is a surface of finite type. By the tameness theorem~\cite{agol2004tameness,cgTameness}, $H$ is tame. Since $F$ is disjoint from some boundary component of $\widehat{M}  $, it is not a fiber of $\widehat{M}  $. By the covering theorem~\cite{canary1996covering}, $H$ is geometrically finite. If $\gamma\in \pi_1(F,x)$, we denote by $[\gamma]$ its image in $\psltc$. 
    
    Our next plan is to take two copies $F_1,F_2$ of the above surface $F$, where $F_i$ has boundary components the $\gamma_{i,j}$, \emph{inductively} glue each pair of $\gamma_{1,j}$ and $\gamma_{2,j}$ with an annulus in $T_j$ (where the height of the annulus gluing $\gamma_{1,j+1}$ and $\gamma_{2,j+1}$ depends on the surface obtained from previous steps), and show the closed surface obtained is incompressible. See Figure \ref{fig:construction} for the case of two boundary components.
    \begin{figure}[htbp]
        \centering
        \includegraphics[width=0.6\linewidth]{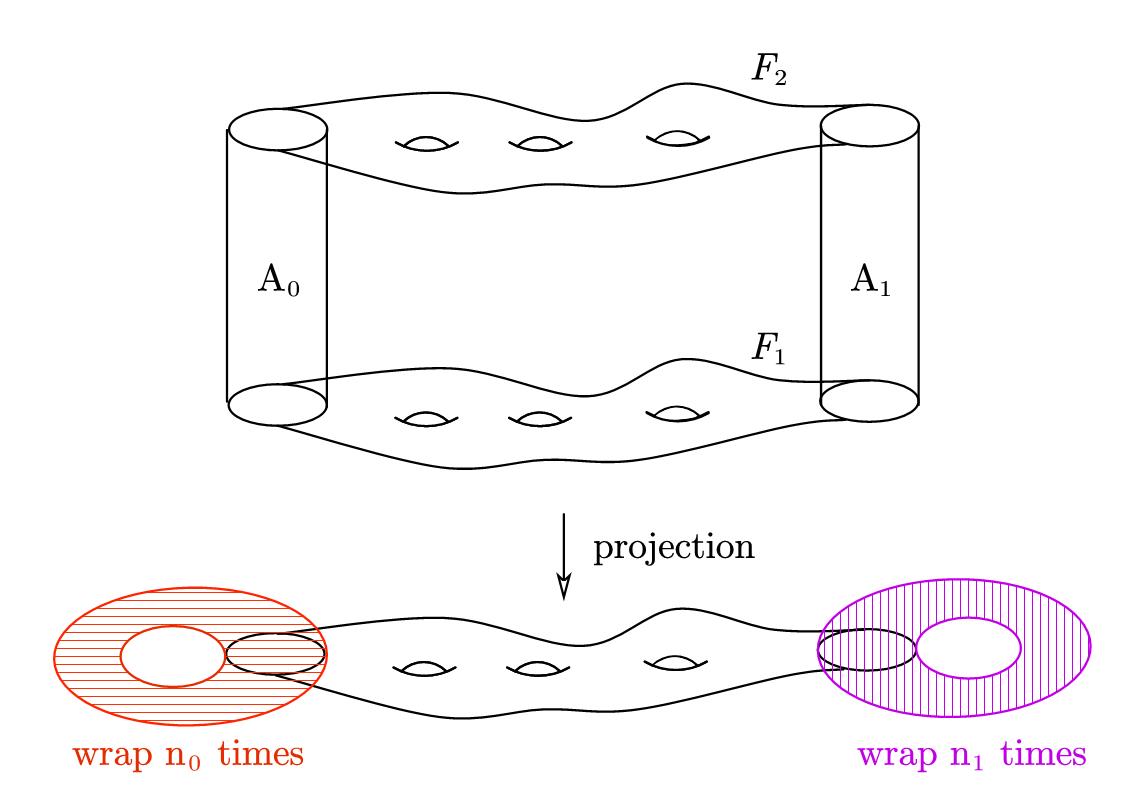}
        \caption{The gluing construction for the case of $m=2$, where $F_1$ and $F_2$ coincide in $M$.}
        \label{fig:construction}
    \end{figure}

    We first glue $F_1,F_2$ together by adding an annulus that connects $\gamma_{1,0}$ and $\gamma_{2,0}$. For $n\in\Z$, let $A_0^n$ be the annulus in $T_0$ with boundary curves $\gamma_{1,0}$ and $\gamma_{2,0}$, that wraps the torus $T_0$ $n$ times in the direction of $\beta_0$; here the basepoint $x$ lies in $\gamma_{1,0}$. We glue $A^n_0$ and $F_2$ along $\gamma_{2,0}$ and obtain a surface $F_2^{0,n}$. We also have
    \begin{equation}
        \pi_1(F_2^{0,n})=\beta_0^n\cdot \pi_1(F_2)\cdot \beta_0^{-n}=\beta_0^n\cdot \pi_1(F_1)\cdot \beta_0^{-n}\cong \beta_0^n H \beta_0^{-n}.
    \end{equation}
    Now we glue $F_1$ and $F_2^{0,n}$ along $\gamma_{1,0}$, and we show that the obtained surface is incompressible. Note that $[\gamma_{1,0}]$ commutes with $\beta_0$ since $[\gamma_{1,0}]\in\langle \alpha_0,\beta_0 \rangle$. 
    \begin{lem}
    	For $n$ sufficiently large, 
    	\begin{equation}
    		H\ast_{\gamma_{1,0}}\bigl(\beta_0^n H \beta_0^{-n}\bigr)     \cong \langle H,\beta_0^n H \beta_0^{-n}\rangle .
    	\end{equation}
    	Hence, $S_0^n$ is incompressible. 
    \end{lem}
	\begin{proof}
		For convenience, we denote $H$ and $\beta_0^n H \beta_0^{-n}$ by $G_1$ and $G_2$ respectively. Let $J$ be the common subgroup of $G_1$ and $G_2$ that is generated by the curve $\gamma_{1,0}$, i.e., $J=\langle[\gamma_{1,0}]\rangle \cong \mathbb{Z}$. Then
		$J$ is a rank-1 parabolic group, corresponding to the boundary curves of $F_1$ and $F_2^{0,n}$ in $T_0$.
		
		Consider the compactified complex plane $\hat{\C}=\C\cup\{\infty\}$ as the boundary of the upper half-space model of $\mathbb{H}^3$. 
		Up to conjugation, we assume that $\infty$ is the parabolic fixed point of $\langle \alpha_0,\beta_0 \rangle$, and the action of $[\gamma_{1,0}]$ on $\C$ is a horizontal translation $z\mapsto z+a$ for some $a\in\R$. Let $z\mapsto z+b$ be the action of $\beta_{0}$ for some complex number $b$ with nonzero imaginary part. Without loss of generality, we assume that $\im(b)>0$.
		
		For $i=1,2$, let $\Lambda_i$ be the limit set of $G_i$ in $\hat{\C}$, which contains $\infty$. 
		Since $G_i$ is geometrically finite and $\infty$ is a rank-1 parabolic fixed point, by \cite[Remark 1.3.2]{keen1991geometric}, $\infty$ is a doubly-cusped parabolic fixed point. Thus there are two disjoint disks $D^1_i,D^2_i$ in $\hat{\C}$ tangent at $\infty$ such that
		\begin{itemize}
			\item $D_i^1\cup D_i^2$ is precisely invariant under $J=\langle [\gamma_{1,0}]\rangle$ in $G_i$;
			\item $\Lambda_i$ is disjoint from $D_i^1\cup D_i^2$.
		\end{itemize}
		
		Since $[\gamma_{1,0}]$ is a horizontal translation on $\C$, the $\partial D_i^j$ are horizontal lines. Hence $\Lambda_i$ is contained in a closed horizontal strip $P_i$ and $\hat{\C}-P_i=D_i^1\cup D_i^2$ is precisely invariant under $J$ in $G_i$. 
		Moreover, each $\Lambda_i$ is connected, by \cite[Section 2.4]{anderson2000cores}, and $\Lambda_i$ is invariant under the horizontal translation $[\gamma_{1,0}]$. Thus $D^1_i$ and $D^2_i$ lie in different components of $\hat{\C}-\Lambda_i$. Let 
		$D_i^1=\{z\in\C: \im(z)<x_i\}$ and $D_i^2=\{z\in\C: \im(z)>y_i\}$ for some real numbers $x_i<y_i$; then 
		$P_i=\{z\in \mathbb{C}: x_i\leq \mathrm{Im}(z)\leq y_i\}$.
		
		Since $G_2$ is conjugate to $G_1$ by $\beta_0^n$, $\Lambda_2$ is obtained by the action of $\beta_0^n$ on $\Lambda_1$. Furthermore, we can assume that $x_2=x_1+n\cdot\im(b)$ and $y_2=y_1+n\cdot\im(b)$.
		Hence there exists an integer $N_0$ such that when $n>N_0$, we have $x_1<y_1<x_2<y_2$.
		
		Pick a positive integer $n_0>N_0$, and let $S_0$ be the surface obtained by gluing $F_1$ and $F_2^{0,n_0}$ along $\gamma_{1,0}$. We claim that $S_0$ is \poi in $\widehat{M}  $.
		Choose an arbitrary real number $s\in(y_1,x_2)$ and let $l=\{z:\im(z)=s\}$ be a horizontal line. Note that $l$ divides $\hat{\C}$ into two disks $B_1:=\{z\in\C: \im(z)>s\}$ and $B_2:=\{z\in\C: \im(z)<s\}$, and we have $P_1\subset B_2$ and $P_2 \subset B_1$.
		
		For $i\in\{1,2\}$, we verify the requirement of \Cref{1st KCT}. For all $g\in J$, the action of $g$ is a horizontal translation, and thus $g(B_i)=B_i$.
		For $g\in G_i\setminus J$, we assert that $g(B_i) \cap B_i=\varnothing$. 
		Indeed, since $B_i\cap P_i=\varnothing$, we have $B_i\subset D^i_1\cup D^i_2$. Then since $g(D^i_1\cup D^i_2)\cap(D^i_1\cup D^i_2)=\varnothing$, we have $g(B_i)\cap B_i=\varnothing$.
		Therefore, $B_i$ is precisely invariant under $J$ in $G_i$. 
		By \Cref{1st KCT} (Maskit's first Kleinian combination theorem), we obtain a new Kleinian group
		\begin{equation}\label{eq:first-amalg}
			K_0=\langle G_1,G_2\rangle \cong G_1 *_{J} G_2.
		\end{equation}
		in $\pi_1(\widehat{M}  )<\PSL(2,\C)$.
		Moreover, $K_0$ is also geometrically finite.
		By the Seifert–van Kampen theorem, the right-hand side of \eqref{eq:first-amalg} is isomorphic to $\pi_1(S_0)$. Thus $S_0$ is an essential surface with $K_0$ the associated subgroup of $\pi_1(\widehat{M}  )$.
	\end{proof}

If $m=1$, i.e., $F$ only has one boundary component, $S_0$ is a closed surface with negative Euler characteristic. Otherwise, we use induction to keep gluing.

For $1\leq k\leq m-1$, suppose we already have an incompressible surface $S_{k-1}$ with boundary by gluing an annulus $A_j$ to connect $\gamma_{j,1}$ in $F_1$ and $\gamma_{j,2}$ in $F_2$, for all $0\leq j\leq k-1$, and a geometrically finite Kleinian group $K_{k-1}$ associated to $S_{k-1}$. Now we glue an annulus $A_{k}$ connecting $\gamma_{k,1}$ and $\gamma_{k,2}$ on $S_{k-1}$. In $K_0$ (and hence in $K_{k-1}$), 
\begin{equation}
    [\gamma_{k,2}]=\beta_0^{n_0}[\gamma_{k,1}]\beta_0^{-n_0}.
\end{equation}
 Let $J_i=\langle[\gamma_{k,i}]\rangle$ with fixed point $p_i\in\hat{\C}$, $i=1,2$. Here $J_2$ is conjugate to $J_1$ in $\pi_1(\widehat{M}  )$ but not in $K_{k-1}$. For $n\in\Z$, let $f_n=\beta_0^{n_0}\beta_k^{-n}$, and we have
\begin{equation}
    f_n^{-1}\cdot [\gamma_{k,2}]\cdot f_n=\beta_k^{n}\beta_0^{-n_0}\cdot\beta_0^{n_0}[\gamma_{k,1}]\beta_0^{-n_0}\cdot\beta_0^{n_0}\beta_k^{-n}=\beta_k^{n}[\gamma_{k,1}]\beta_k^{-n}=[\gamma_{k,1}]
\end{equation}
since $\gamma_{k,1} \in \langle \alpha_k,\beta_k \rangle$, which is a parabolic subgroup. 
Hence $f_n\cdot J_1\cdot f_n^{-1}=J_2$.

Now we find precisely invariant disks for $J_i$ as follows.
Let $\Omega$ be the domain of discontinuity of $K_{k-1}$. Since $K_{k-1}$ is geometrically finite, $\Sigma=\Omega/K_{k-1}$ is a (possibly disconnected) surface of finite type by Ahlfors finiteness theorem~\cite{alFinitelyGeneratedKleinian}. Let $\pi\from \Omega\to\Sigma$ be the covering map. Since each $p_i$ is a rank-$1$ parabolic fixed point, it is doubly cusped and hence projects to a pair of punctures $\hat{p}_i^j$, $j=1,2$, on $\Sigma$. We claim that the $\hat{p}_i^j$ are all distinct on $\Sigma$. Otherwise, there exists $g\in K_{k-1}$ such that $g(p_1)=p_2$, and then $g$ conjugates $J_1$ to $J_2$. Hence $\gamma_{k,1}$ is freely homotopic to $\gamma_{k,2}$ in the surface $S_{k-1}$, which is impossible.

Since $\Sigma$ is Hausdorff, there exists disjoint open neighborhoods $\hat{U}_i^j$ of $\hat{p}_i^j$. Let $U_i^j\subset\Omega$ be the connected component of $\pi^{-1}(\hat{U}_i^j)$ that contains $p_i$ on its boundary. Then each $U_i^j$ is a disk at $p_i$ precisely invariant under $J_i$ in $K_{k-1}$, and all $U_i^j$ are disjoint. 
By shrinking $\hat{U}_1^j$ and relabeling if necessary, we assume that $\beta_0^{n_0}(U_1^j)\subset U_2^j$, for $j=1,2$. 
Since $[\gamma_{k,1}]$ is not a multiple of $\beta_k$, there exists a positive integer $N_k$ such that when $n>N_k$, $\beta_k^{-n}$ maps the exterior of $U_1^j$ into the interior of $U_1^{j+1}$ for exactly one $j\in\{1,2\}$; here $U_1^3=U_1^1$. Without loss of generality, we assume that $\beta_k^{-n}$ maps  $\hat{\mathbb{C}}-U_1^1$ into the interior of $U_1^2$.
Choose an arbitrary integer $n_k>N_k$. Let $B_1=U_1^1$ and $B_2=f_{n_k}((\overline{U_1^1})^c)=\beta_0^{n_0}\circ \beta_k^{-n_k}((\overline{U_1^1})^c)$. We show that $B_1,B_2$ and $f_{n_k}$ satisfy the requirements of \Cref{2nd KCT}:
\begin{enumerate}[label=(\roman*)]
    \item If there exists $g\in K_{k-1}$ such that $g(B_1)\cap B_2\neq\varnothing$, let $z\in g(B_1)\cap B_2$. Since 
    $$z\in B_2 =\beta_0^{n_0}\circ \beta_k^{-n_k}((\overline{U_1^1})^c)\subset \beta_0^{n_0}(U_1^2)\subset U_2^2,$$ 
    we have $\pi(z)\subset\hat{U}_2^2$. On the other hand, since $z\in g(B_1)$, we have $g^{-1}(z)\in B_1$, and therefore $\pi(z)=\pi(g^{-1}(z))\in\hat{U}_1^1$. This contradicts the assumption that  $\hat{U}_1^1\cap \hat{U}_2^2=\varnothing$. Hence for any $g\in K_{k-1}$, $g(B_1)\cap B_2=\varnothing$.

    \item Since $J_i$ is the parabolic group that fixes $p_i$, each $U_i^j$ is precisely invariant under $J_i$ in $K_{k-1}$, $j=1,2$. Thus $B_1$ is precisely invariant under $J_1$ in $K_{k-1}$. Since $B_2$ is a topological disk contained in $U^2_2$ and $\partial B_2=\beta_0^{n_0}\circ \beta_k^{-n_k}(\partial U_1^1)$ is invariant under $J_2$, $B_2$ is invariant under $J_2$. For $g\in K_{k-1}-J_2$, $g(B_2)\subset g(U^2_2)$, which is disjoint from $U^2_2$. Hence $g(B_2)\cap B_2=\varnothing$.

    \item By definition, we have $f_{n_k}(\partial B_1)=\partial B_2$, and $f_{n_k}$ maps the exterior of $B_1$ onto the interior of $B_2$.
\end{enumerate}
Therefore by \Cref{2nd KCT} (Maskit’s second Kleinian combination theorem), we have
\begin{equation}
    K_k:=\langle K_{k-1}, f_{n_k}\rangle\cong K_{k-1}\ast_{f_{n_k}},
\end{equation}
which is geometrically finite.

Let $A_k$ be an annulus in $T_K$ with boundary curves $\gamma_{k,1}$ and $\gamma_{k,2}$, that wraps the torus $T_K$ of $n_k$ times in the direction of $\beta_k$; here $\gamma_{k,1}$ and $\gamma_{k,2}$ are different boundary curves of $S_{k-1}$, having the same image in $\widehat{M}  $. We then glue $A_k$ to the surface $S_{k-1}$, to connect $\gamma_{k,1}$ and $\gamma_{k,2}$, and obtain a surface $S_k$. Hence $\pi_1(S_k)\cong K_{k-1}\ast_{f_{n_k}}\cong \langle K_{k-1}, f_{n_k}\rangle$ and $S_k$ is incompressible in $\widehat{M}  $.

Since $F$ only has finitely many boundary components, we obtain an incompressible closed surface $S_{m-1}$ in $\widehat{M}  $ with negative Euler characteristic. 
The boundary elements $[\gamma_i]$ of $F$ are all parabolic elements in $\pi_1(\widehat{M})$, hence $S_{m-1}$ is coannular.
Finally, since $F$ is disjoint from some torus boundary component of $\widehat{M}  $, so is $S_{m-1}$.
\end{proof}

We construct infinitely many homotopy classes of coannular surfaces of the same genus in $\widehat{M}$ using the existence of coannular surfaces constructed above.

\begin{defi}[Spin a coannular surface along a boundary torus]\label{def:spin}
    Let $f\from S\to \widehat{M}  $ be a closed, oriented, $\pi_1$-injective surface that is coannular to some toroidal boundary component $T$ with coannular slope $\alpha$.
    We cut $f(S)$ and $T$ along $\alpha$, obtaining two surfaces with boundary, denoted by $(\overline{f(S)}, \alpha_{f}^+, \alpha_{f}^-)$ and $(\overline{T}, \alpha_{T}^+, \alpha_{T}^-)$. We then glue $\overline{f(S)}$ and $\overline{T}$ along their boundaries by identifying $\alpha_{f}^+$ with $\alpha_{T}^-$ and $\alpha_{f}^-$ with $\alpha_{T}^+$ and get a new surface with the same genus.
    We say that the new surface is obtained by \emph{spinning} $f$ \emph{around} $T$ \emph{along} $\alpha$, and denote the resulting map by $f_T^1\from S\to \widehat{M}  $.
\end{defi}

If $f_T^1$ is $\pi_1$-injective, then it has common parabolics as $f$ in $\widehat{M}  $; thus we can spin $f_T^1$ around $T$ along $\alpha$, and obtain $f_T^2$. By iterating, we get a sequence of surfaces $\{f_T^k\}_{k=0}^\infty$ with $f_T^0=f$.
The following lemma proves that, starting from the surface obtained in \Cref{thm6.5}, we can spin this surface along the last torus boundary $T_m$, and get a sequence of non-homotopic $\pi_1$-injective surfaces.

\begin{lem}\label{lemma6.4}
    Suppose $\widehat{M}  $ is a cusped hyperbolic $3$-manifold with at least 3 torus boundary components. Let $f\from S\to \widehat{M}  $ be an incompressible closed surface obtained in \Cref{thm6.5} that is coannular to torus boundary components $T_0,\ldots, T_{m-1}$ of $\widehat{M}  $, and $S$ is disjoint from at least one torus boundary component. Then for any positive integers $k,j$ with $k\ne j$, $f_{T_{m-1}}^k$ is not homotopic to $f_{T_{m-1}}^j$.
\end{lem}

\begin{proof}
    By the proof of \Cref{thm6.5}, we know that spinning around the last torus boundary component $T_{m-1}$ will always generate incompressible surfaces. Hence it suffices to show that they are pairwise non-homotopic.

    Since $\widehat{M} $ has at least 3 torus boundary components, for each $i$, the torus boundary component $[T_i]$ is nonzero in $H_2(\widehat{M}  ;\mathbb{Z})$. The spin applied to $S$ around $T_{m-1}$ at the level of $H_2$ amounts to $[S]+[T_{m-1}]$.
    By $[T_{m-1}]\neq0$, when $k\neq j$, 
    $$[f_{T_{m-1}}^k(\tilde{S})]=[\tilde{S}]+k[T_{m-1}]\neq [\tilde{S}]+j[T_{m-1}]=[f_{T_{m-1}}^j(\tilde{S})],$$ 
    which implies that $f_{T_{m-1}}^k$ and $f_{T_{m-1}}^j$ are not homotopic. 
\end{proof}

\begin{proof}[Proof of Theorem \ref{thm:coannular}.]
By \Cref{lem:clrlemma2.1}, we first pass to a finite cover $\widehat{M}  $ of $M$ so that $\widehat{M}  $ has at least $3$ torus boundary components. 
By \Cref{thm6.5} and \Cref{lemma6.4}, we have a sequence of incompressible closed coannular surfaces $f^i\from S\rightarrow \widehat{M}$ for $i\geq 0$ such that $f^k$ and $f^j$ are not homotopic whenever $k\ne j$.

Since $\widehat{M}  $ is a finite cover of $M$, only finitely many distinct conjugacy classes in $\pi_1(\widehat{M})$ project down to the same conjugacy class in $\pi_1(M)$. Therefore, we obtain infinitely many coannular surface subgroups of genus $g$ up to conjugacy in $\pi_1(M)$.
\end{proof}

\begin{rem}
The number of embedded essential closed surfaces of bounded genus up to homotopy in a cusped hyperbolic $3$-manifold is finite, which follows from a combination of [\citenum{tjIsotopyIncompressible}, Theorem 4.5]  and [\citenum{dgrCountingEssential}, Lemma 4.4]. 

In \cite{ptPolynomialsBoundsCusped}, Purcell and Tsvietkova gave explicit polynomial bounds for the number of \poi embedded surfaces, possibly with boundary, in alternating link complements. We refer to \cite{dgrCountingEssential,ptPolynomialsBoundsCusped}
and references therein for related work on the number of \poi embedded surfaces in $3$-manifolds. 
\end{rem}